\documentclass[12pt]{amsart}
\usepackage{amsmath}%
\usepackage{amsfonts}%
\usepackage{amssymb}%
\usepackage{graphicx}
\usepackage{tikz-cd}
\usepackage{enumerate}
\usepackage{rotating}
\usepackage{rotfloat}
\usepackage{caption}

\newtheorem{thm}{Theorem}[section]
\newtheorem{prop}[thm]{Proposition}
\newtheorem{cor}[thm]{Corollary}
\newtheorem{lemma}[thm]{Lemma}
\theoremstyle{cupdefn}
\newtheorem{defn}[thm]{Definition}
\theoremstyle{cuprem}
\newtheorem{rem}[thm]{Remark}
\numberwithin{equation}{section}

\newcommand{\R}{\mathbb{R}}
\newcommand{\N}{\mathbb{N}}
\newcommand{\Z}{\mathbb{Z}}
\newcommand{\C}{\mathbb{C}}

\floatstyle{plain}
\floatname{diagram}{Diagram}
\newfloat{diagram}{tbp}{lop}[section]

\begin{document}

\title[Cohomology of orbit of involutions on Wall manifolds]{Cohomology algebra of orbit spaces of free involutions on some Wall manifolds}

\author[T. F. V. Paiva]{Thales Fernando Vilamaior Paiva}
\address{Universidade Federal de Mato Grosso do Sul, C\^ampus de Aquidauana \\ CEP 79200-000, Aquidauana - MS, Brazil }
\email{thales.paiva@ufms.br}

\author[E. L. dos Satos]{Edivaldo Lopes dos Santos}
\address{Departamento de Matem\'atica\\
	Centro de Ci{\^e}ncias Exatas e Tecnologia\\
	Universidade Federal de S{\~a}o Carlos \\
	CP 676, CEP 13565-905, S\~ao Carlos, SP, Brazil}
\email{edivaldo@dm.ufscar.br}

\thanks{The second author was supported by the Projeto Tem{\'a}tico: Topologia Alg{\'e}brica, Geom{\'e}trica e Diferencial, FAPESP Process Number 2018/17240-8.}




\begin{abstract} 
In this paper, we investigate the existence of free involutions on some Wall manifolds and we compute the mod 2 cohomology algebra of the correspondent orbit space.
\end{abstract}

\maketitle

\section{introduction}

Let $X$ be a topological space and let $G$ be a topological group. A relevant question associated to a pair $(X,G)$ is to determine when there is a continuous free action of $G$ on $X,$  as the classical example proposed by Hopf in \cite{hopf} of classify all manifolds with universal cover homeomorphic to the $n-$dimensional sphere $S^{n},$ which is equivalent to determine all finite groups that can act freely on $S^{n}.$ In \cite{milnor57} J. Milnor provided an answer in this direction, showing that the symmetric group $\mathbb{S}_{3}$ cannot act freely on $S^{n}.$ A possible generalization can be made considering $X$ as a product of spheres, as proposed by Dotzel et al. in \cite{ronald} (for free actions of $\Z_{p}$ on $S^{m}\times S^{n}$) and by Jahren and Kwasik in \cite{jahren} (for free involutions on $S^{1}\times S^{n}$). 

When $G$ acts freely on $X$ a question that arises is about the nature of the corresponding orbit space, as in the standard examples: the real, complex and quaternionic projective spaces, which are orbit spaces of standard actions of $\Z_{2},$ $S^{1}$ and $S^{3}$ on spheres, respectively. The general examples that may appear suggest that calculating the cohomology of these orbit spaces can be complicated, particularly for compact manifolds other than spheres, as done, for example, in the work of Pergher et al. \cite{Pergher} about involutions on spaces of cohomology type $(a,b)$ and Singh \cite{singh13} for involutions on lens spaces. This paper concerns this question for $X$ a Wall manifold of the type $Q(m,n)$ and $G=\Z_{2}.$

The Wall manifolds were introduced by C. T. C. Wall in \cite{wall60}, where it was shown that these manifolds give generators in odd dimensions for the unoriented bordism ring, as occur with the Dold manifolds and the Milnor manifolds, as shown by Dold in \cite{dold56} and Milnor in \cite{milnor65}, respectively. Therefore, it is a relevant question to determine algebraic invariants of these manifolds, as done for example by Mukerjee in \cite{Mukerjee03, Mukerjee04, Mukerjee06}.

For free actions of $\Z_{2}$ (or equivalently free involutions) on Dold ma\-nifolds, the problem was answered in 2018 by Morita et al. in \cite{ana} for the cases $P(1,\text{odd}),$ and for all cases in 2019 by Dey in \cite{pinka} using slight different techniques. About the Milnor manifolds, the problem was solved in 2019 by Dey and Singh in \cite{DeySingh}. The main result that we present in this paper is the complete description of the cohomology of the orbit space of free involution on a Wall manifold of the type $Q(1,\text{odd}),$ according to Theorem \ref{main}.

\begin{thm}\label{main}
Let us consider $G=\Z_{2}$ acting freely on a compact Hausdorff space $X$ with mod 2 cohomology of a Wall manifold $Q(1,n),$ with $n\geq 3$ odd and such that the induced action on the mod 2 cohomology is trivial. Then 
$$H^{\ast}(X/G;\Z_{2})\cong \frac{\Z_{2}[\alpha,\beta,\gamma,\delta]}{\langle \alpha^{3}, \beta^{3}, \gamma^{2},\beta^{2}+\beta\gamma, \delta^{\frac{n+1}{2}}\rangle},$$ 
where $\deg \alpha=\deg \beta=\deg \gamma=1$ and $\deg \delta=4.$
\end{thm}

The choice of the case $Q(1,n),$ as an initial stage, is motivated by the work of Khare in \cite{Khare}, where he shows that $Q(m,n)$ bounds, if and only if, $n$ is odd or $n	=0$ and $m$ is odd. Therefore, these spaces can admit a free involution and particularly the choice $m=1$ is convenient, because of the relationships established between the generators of the cohomology ring, while for arbitrary values of $m$ may be very difficult to find obstructions for the large numbers of possible cohomology ring structures. 

For this purpose this paper is organized as follows. In Section \ref{Section2}, we established the results on spectral sequences and the Borel fibration which are the main tools that we use. In Section \ref{Section3}, we define the Wall manifolds, as a mapping torus of a Dold manifold, and do some hypotheses about the dimension to ensure that these manifolds admit a free involution and that the system of local coefficients induced by Borel fibration is simple. The Section \ref{Section4} is dedicated to the proof that there is only one possible homotopy class of the orbit space of free involutions on some Wall manifolds and in Section \ref{Section5}, we proved the main Theorem \ref{main}.

\section{Preliminaries}\label{Section2}


\indent\indent Let $G$ be a compact Lie group, and let $X$ be a paracompact $G-$space. We denote by $G\hookrightarrow E_{G}\rightarrow B_{G}$ the universal $G-$bundle, as in \cite{milnor56}, and by 
\begin{equation}
\begin{tikzcd}
X \arrow[hookrightarrow]{r} & X_{G} \arrow{r}{\pi} & B_{G} 
\end{tikzcd}
\end{equation}
the Borel fibration (\cite{dieck87}, Chapter 3) associated, where\linebreak $X_{G}=(X\times E_{G})/G$ is the orbit space of the diagonal action and $\pi$ is induced by the projection $X\times E_{G}\to E_{G}.$ 

The projection at the first coordinate $X\times E_{G}\to X$ is $G-$equivariant, therefore it induces the principal $G-$bundle
\begin{equation}
\begin{tikzcd}
E_{G} \arrow[hookrightarrow]{r} & X_{G} \arrow{r}{p} & X/G.
\end{tikzcd}
\end{equation}

When $G$ acts freely on $X,$ the fiber $E_{G}$ is contractible so $p$ is a homotopy equivalence (\cite{dieck87}, p. 180), which induces a natural isomorphism 
\begin{equation}\label{is}
H^{\ast}(X_{G};R)\cong H^{\ast}(X/G;R),
\end{equation} 
for all commutative ring with unit $R.$ 

In this context, by the Leray-Serre spectral sequence (\cite{mcleary01}, Theorem 5.2) there is a first quadrant cohomological spectral sequence $\{E^{\ast,\ast}_{r},d_{r}\}$ converging to $H^{\ast}(X_{G};R)\cong H^{\ast}(X/G;R),$ as an algebra, such that 
\begin{equation}\label{ss}
E_{2}^{p,q}\cong H^{p}(B_{G};\mathcal{H}^{q}(X;R)),
\end{equation}
the cohomology of $B_{G}$ with local coefficients in the cohomology of the fiber of $\pi.$ Precisely, $\mathcal{H}^{\ast}(X;R)$ denotes the cohomology of $X$ twisted by the action of the fundamental group of $B_{G}.$  

When $\pi_{1}(B_{G})$ acts trivially on $H^{\ast}(X;R)$ the system of local\linebreak coefficients is simple (\cite{mcleary01}, Proposition 5.6) then (\ref{ss}) takes the sui\-table form 
\begin{equation}\label{pt}
E_{2}^{p,q}\cong H^{p}(B_{G};R)\otimes_{R} H^{q}(X;R).
\end{equation}

Besides that by (\cite{mcleary01}, Theorem 5.9), the homomorphisms
\begin{equation}\label{hom1}
H^{q}(B_{G};R)=E_{2}^{q,0}\twoheadrightarrow \cdots\twoheadrightarrow E_{q}^{q,0}\twoheadrightarrow E_{q+1}^{q,0}=E_{\infty}^{q,0}\subseteq H^{q}(X_{G};R)
\end{equation}	
and
\begin{equation}\label{hom2}
H^{q}(X_{G};R)\twoheadrightarrow E_{\infty}^{0,q}=E_{q+1}^{0,q}\subseteq E_{q}^{0,q}\subseteq \cdots\subseteq E_{2}^{0,q}=H^{q}(X;R)\end{equation}
are equal to the induced homomorphisms 
$\pi^{\ast}:H^{q}(B_{G};R)\to H^{q}(X_{G};R)$ and $i^{\ast}:H^{q}(X_{G};R)\to H^{q}(X;R),$ respectively.

\begin{rem}
\emph{To simplify the notations, throughout the text we will always denote by $H^{\ast}(\text{ - })= \check{H}^{\ast}(\text{ - };\Z_{2})$ the $\check{\text{C}}$ech cohomology with coefficient in} $\Z_{2}$ \rm{(}\cite{bredon}, Section 3.6\rm{)}, and we will indicate the cup product $a\smile b$ simple by $a\cdot b=ab.$
\end{rem}

\begin{thm}[\cite{bredon}, p. 374, Theorem 1.5]\label{bredon2}
Let $X$ be a paracompact Hausdorff space such that $H^{j}(X)=0,$ for all $j> n,$ for some $n\in\N.$ If $\Z_{2}$ acts freely on $X,$ then $H^{j}(X/\Z_{2})=0,$ for all $j>n.$
\end{thm}

We will use this results to get the cohomology algebra of the orbit space $X/G,$ where $G=\Z_{2}$ and $X$ is a Wall manifold of the type $Q(m,n)$ to certain natural numbers $m$ and $n,$ which we will define posteriorly.


\section{Wall manifolds}\label{Section3}


Let $m$ and $n$ be non negative integers. In 1956 A. Dold \cite{dold56} defined the differentiable manifolds $P(m,n),$ as the orbit space of the free involution
\begin{equation}\label{dd1}
\begin{array}{rccc}
T:&S^{m}\times \C P^{n}&\to & S^{m}\times\C P^{n}\\
  & (x,[z])& \mapsto & (-x,[\overline{z}])\\
\end{array}. 
\end{equation}

Dold further showed that the mod 2 cohomology ring of $P(m,n)$ is a graduated polynomial ring in the variables $c\in H^{1}(P(m,n))$ and $d\in H^{2}(P(m,n)),$ truncated by the relations $c^{m+1}=0$ and $d^{n+1}=0.$

In 1960, Wall defined in \cite{wall60} a  new set of generators to $\mathfrak{R}_{\ast},$ using the following general construction: Let $M$ be a closed differentiable manifold of dimension $n$ and let $f:M\to M$ be a diffeomorphism. On the cartesian product $M\times [0,1]$ we define the equivalence relation $(x,s)\sim_{f} (y,t)\Leftrightarrow y=f(x),$ for $s=0$ and $t=1,$ or $(x,s)=(y,t)$ otherwise, which identifies only the boundary points of $M\times [0,1].$ The quotient space of $M\times [0,1]$ by relation $\sim_{f}$ is a closed differentiable manifold of dimension $n+1,$ which we will denote by $\mathcal{W}(M,f).$

The free involution $R\times 1:S^{m}\times \C P^{n}\to S^{m}\times \C P^{n},$ where\linebreak $R:S^{m}\to S^{m}$ is the reflection of the last coordinate given by\linebreak $R(x_{0},\cdots,x_{m})= (x_{0},\cdots,x_{m-1},-x_{m})$ and $1:\C P^{n}\to \C P^{n}$ is the identity, commutes with the involution $T$ given in (\ref{dd1}), so it induces a free involution $S:P(m,n)\to P(m,n).$ 

\begin{defn}
For each pair of non negative integers $m,n,$ the space $\mathcal{W}\left(P(m,n),S\right)$ is a closed differentiable manifold of dimension\linebreak $m+2n+1,$ called Wall manifold and denoted by $Q(m,n).$ Alternatively, $Q(m,n)$ is defined as the mapping torus of the homeomorphism $S:P(m,n)\to P(m,n),$ that is, 
\begin{equation}
Q(m,n)=\dfrac{P(m,n)\times [0,1]}{([x,z],0)\sim(S[x,z],1)}.
\end{equation}
\end{defn}

\begin{prop}[\cite{wall60}, Lemma 4]
The mod 2 cohomology ring of $Q(m,n)$ is polynomially generated by elements $x,c\in H^{1}(Q(m,n))$ and $d\in H^{2}(Q(m,n)),$  truncated by the relations $x^{2}=0,$ $d^{n+1}=0$ and $c^{m+1}=c^{m}\cdot x,$ that is,
\begin{equation}
H^{\ast}(Q(m,n))=\Z_{2}[x,c,d]/\langle x^{2},c^{m+1}+c^{m}x,d^{n+1} \rangle.
\end{equation}
\end{prop}

\begin{defn}\label{def1}
A Wall structure on a manifold $M$ is a homotopy class of maps $\beta:M\to S^{1}$ such that $\beta^{\ast}(\iota)=w_{1}(M)$ the first Stiefel-Whitney class, where $\iota\in H^{1}(S^{1})$ is the generator.
\end{defn}

The projection $P(m,n)\times [0,1]\to [0,1]$ induces a fibration \linebreak$\beta:Q(m,n)\to S^{1}$ such that $\beta^{\ast}(\iota)=x=w_{1}(Q(m,n))$ and this provides a Wall structure on the manifolds $Q(m,n)$ as in the Definition \ref{def1}. In \cite{wall60}, Wall showed through these constructions that 
\begin{equation}
\{[\R P^{2i}],[Q(2^{r}-2,s2^{r})]\text{ };\text{ }i,r,s\geq 1	\}
\end{equation} 
is a generator set of $\mathfrak{R}_{\ast}.$ Before that, in 1956, Dold in \cite{dold56} has shown that 
\begin{equation}
\{[\R P^{2i}],[P(2^{r}-1,s2^{r})]\text{ };\text{ }i,r,s\geq 1	\}
\end{equation} 
is a generator set of the bordism ring, and in 1965 Milnor showed in  \cite{milnor65} that 
\begin{equation}\{[\R P^{2i}],[H(2^{k},2t2^{k})]\text{ };\text{ }i,k,t\geq 1	\}
\end{equation} 
is also a generator set of $\mathfrak{R}_{\ast},$ where $H(m,n)$ is a (m+n-1)-dimensional Milnor manifold. Therefore it is a relevant question to get invariants of these manifolds.

\begin{rem}
If a topological space $X$ admits a free action of $\Z_{2}=\langle g \rangle$ then this induces a free involution $T:X\to X$ defined by\linebreak $T(x)=g\ast x.$ On other hand, the involution $T$ induces a $\Z_{2}-$action on the cohomology ring of $X,$ given by $T^{\ast}:H^{\ast}(X)\to H^{\ast}(X).$ 
\end{rem}

From the bordism theory (\cite{Conner}, Theorem 24.2) we know that if a closed differentiable manifold $X$ does not bounds, so it cannot admits a free involution. However, when $X$ bounds there may or may not exists such action. In \cite{Khare}, S. S. Khare shows that  $Q(m,n)$ bounds, if and only if, $n$ is odd or $n=0$ and $m$ odd. Therefore, for the particular cases where  $n\geq 1$ is an odd integer such manifolds may or may not admits a free $\Z_{2}-$action.

\begin{prop}\label{exist}
For all odd integer $n\geq 1$, the Wall manifold $Q(m,n)$ admits a free involution.
\end{prop}

\begin{proof}
Let $\varphi:P(m,n)\times [0,1]\to P(m,n)\times [0,1]$ the free involution defined by 
$$\varphi \left([x,[z_{0}:\cdots :z_{n}]],s\right)=\left([-x,[-\overline{z}_{1}:\overline{z}_{0}:\cdots : -\overline{z}_{n}:\overline{z}_{n-1}]],s\right).$$

If $a=\left([x,[z]],s\right)\sim_{S} \left([y,[w]],t\right)=b,$ that is, $s=0,$ $t=1$ and 
$$\left[y,[w]\right]=S\left[x,[z]\right]=\left[(x_{0},\cdots,x_{m-1},-x_{m}),[\overline{z}]\right],$$
then $\varphi(a)\sim_{S}\varphi(b).$ Therefore, $\varphi$ induces a free involution $F$ on $Q(m,n).$
\end{proof}

For $X=Q(1,n),$ we have the relations $d^{n+1}=0,$ $c^{2}=c\cdot x$ and $x^{2}=0,$ which provides the following description of cohomology the groups:
$$
\begin{array}{rcclcl}
H^{2k+1}(X) &\cong& \langle c\cdot d^{k}\rangle\oplus\langle x\cdot d^{k}\rangle&\cong& \Z_{2}^{2}; & 0\leq k\leq n ;\\
H^{2k}(X) & \cong & \langle d^{k} \rangle\oplus\langle d^{k-1}\cdot c^{2}\rangle&\cong& \Z_{2}^{2}; & 1\leq k\leq n, 
\end{array}
$$
and $H^{2n+2}(X)\cong \Z_{2}\langle d^{n}\cdot c^{2}\rangle,$ where $2n+2$ is the dimension of $X.$ Therefore, even when $m=1,$ it is not immediate that an action of $\Z_{2}$ on $H^{\ast}(Q(m,n))$ has to be necessarily trivial.

By Proposition \ref{exist}, we can ask about the nature of the orbit space of some free $\Z_{2}-$action in a Wall manifold of the type $Q(m,n),$ for $n$ an odd natural. We will show, in Section \ref{Section4}, that there is an unique  mod 2 cohomology ring $H^{\ast}(Q(m,n)/\Z_{2}),$ when $m=1$, whose structure is given in Theorem \ref{main}.  For the others values of $m,$ we have at least an estimate for the number of possible cohomology ring structures. 

To obtain this results, we will use the spectral sequence associated to the Borel fibration (\ref{bof}) and the equality (\ref{pt}). For that we need to impose some conditions to ensure that the action of $\pi_{1}(B_{\Z_{2}})=\pi_{1}(\R P^{\infty})=\Z_{2}$ in the cohomology of the fiber $Q(m,n)$ is trivial, what we do in the next results. 

\begin{rem}
If there is an action of $\Z_{2}=\langle g\rangle$ on $H^{\ast}(Q(m,n)),$ as the cohomology ring is generated by the elements $x,c$ and $d,$ this action is entirely determined when we know the operations $g\ast x,$ $g\ast c$ and $g\ast d,$ that is, we just  have to analyze the actions 
\begin{equation}
g^{\ast}:H^{j}(Q(1,n))\to H^{j}(Q(1,n)),\,\, \alpha\mapsto g\ast\alpha, \text{ for } j=1,2,
\end{equation}  
and to analyze it we note that by the naturality of cup product,  
\begin{equation}
g\ast (\alpha^{r}\cdot \beta^{s})=(g\ast \alpha)^{r}\cdot (g\ast \beta)^{s},\,\,\,\forall \alpha,\beta\in H^{\ast}(Q(m,n)),
\end{equation}
for all $r,s\geq 0,$ which will be useful in the next calculations. 
\end{rem}

Clearly, any action of $\Z_{2}=\langle g\rangle$ on $H^{\ast}(Q(m,n))$ must be trivial on the generator $x$, otherwise if $g\ast x=c$ or $g\ast x=c+x,$ then $0=g\ast x^{2}=c^{2},$ which is a contradiction for all $m\geq 1.$ So, we only need to analyze $g\ast c$ and $g\ast d,$ for which there are only the possibilities $g\ast c\in\{c,c+x\}$ and $g\ast d\in \{d,cx,c^{2},cx+d,c^{2}+d,cx+c^{2},cx+c^{2}+d\}.$ 

For $n> 1,$ if $g\ast d=cx,$ then $g\ast d^{2}=(cx)^{2}=0,$ which implies that $d^{2}=0,$ but it is a contradiction. When $m=1,$ then $c^{2}=cx,$ and there is only the possibility $g\ast d\in \{d,cx+d\}.$ In this case, we will see that the possibility $g\ast d =cx+ d$ also cannot to be occurs under some dimensional condition and for that we will use the following general theorem about involutions.

\begin{thm}[\cite{bredon}, p. 407]\label{invo}
Let $T$ be an involution on a finitistic space $X.$ Suppose that $H^{i}(X)=0,$ for $i>2l$ and let $T^{\ast}$ be the identity on $H^{2l}(X).$ If there is an element $a\in H^{l}(X)$ such that $a\cdot T^{\ast}(a)\neq 0$ then the fixed point set of $T$ is nonempty.
\end{thm}

\begin{prop}\label{acaotrivial}
Let $T$ be a free involution in $Q(1,n)$ and let $T^{\ast}$ be the induced homomorphism on the cohomology ring $H^{\ast}(Q(1,n)).$ If $n\equiv 1(\text{mod }4),$ then $T^{\ast}(d)=d.$
\end{prop}

\begin{proof}
Let us consider $\Z_{2}=\langle g\rangle$ and let us suppose that $T^{\ast} (d)=g\ast d= cx+d.$ For $s=(n+1)/2,$ let us take $d^{s}\in H^{n+1}(Q(1,n)).$ As $n\equiv 1(\text{mod }4),$ then $s= \binom{s}{s-1}$ is odd, therefore
$$d^{s}T^{\ast}(d^{s})=d^{s}(cx+d)^{s}= d^{s}(cxd^{s-1}+d^{s})=d^{n}cx\neq 0,$$
since $\langle d^{n}cx\rangle = H^{2n+2}(Q(1,n)).$ Taking $a=d^{s}$ and $l=n+1,$ it follows from Theorem \ref{invo} that the fixed point set of $T$ is nonempty, what is not true. 
\end{proof}

\begin{cor}\label{acaotrivial2}
Let $T$ be a free involution on $Q(1,n)$ and let $T^{\ast}$ be the induced $\Z_{2}-$action on $H^{\ast}(Q(1,n)).$ If $T^{\ast}(c)=c$ and $n\equiv 1(\text{mod }4),$ then $T^{\ast}$ is trivial.
\end{cor}

We observe that even if $c\neq T^{\ast}(c)=c+x,$ we have $T^{\ast}(\alpha)=\alpha,$ for all $\alpha\in H^{i}(Q(1,n)),$ $i>1,$ that is, the action is trivial, except for  the elements $c$ and $c+x$ in the cohomology ring. 

When $m$ is even, then the induced action has to be trivial on the generators $x$ and $c$ of $H^{\ast}(Q(m,n))$, otherwise if $g\ast c=c+x,$ then 
\begin{equation}
g\ast c^{m+1}=(c+x)^{m+1}=c^{m+1}+c^{m}x=0,
\end{equation} 
but we know that $c^{m+1}=c^{m}x\neq 0.$ It still implies that $g\ast d\neq c^{2}$ and $g\ast d\neq c^{2}+cx.$ In fact, as the action induces an automorphism $\varphi$ on the cohomology groups, if $g\ast d=c^{2},$ then $d+c^{2}\in \ker\varphi=\{0\}$ and if $g\ast d = c^{2}+cx,$ then $d+c^{2}+cx\in \ker\varphi=\{0\}$. So, in this case there is only the possibility 
\begin{equation}
T^{\ast}(d)=g\ast d\in\{d,c^{2}+d,cx+d,c^{2}+cx+d\}.
\end{equation}


\section{Uniqueness of cohomology ring}\label{Section4}

 
The main purpose of this section is to show that there is only one possible homotopy class of the orbit space of free involutions on $Q(1,n),$ under some hypotheses on $n.$ For this, let us consider the $G-$space $Q(1,n),$ where $G$ is the group $\Z_{2},$ and 
\begin{equation}\label{bof}
\begin{tikzcd}
Q(1,n) \arrow[hookrightarrow]{r} & Q(1,n)_{G}\arrow{r}{\pi}& B_{G}
\end{tikzcd}
\end{equation}  
is the associated Borel fibration.

Under the hypothesis of Corollary \ref{acaotrivial2}, the group $\pi_{1}(B_{G})=\Z_{2}$ acts trivially on the cohomology of the fiber $Q(1,n),$ so there is a first qua\-drant cohomology spectral sequence $\{E_{r}^{\ast,\ast},d_{r}\},$ where
\begin{equation}\label{seq1}
E_{2}^{p,q}\cong H^{p}(B_{G})\otimes_{G}H^{q}(Q(1,n)),
\end{equation}
which converges as an algebra to $H^{\ast}(Q(1,n)_{G})\cong H^{\ast} \left(Q(1,n)/G\right).$

We will use the spectral sequence whose $E_{2}-$term is given in (\ref{seq1}) to get possibles structures of the cohomology ring $H^{\ast}(X/G),$ through the isomorphism (\ref{is}), where $X=Q(1,n),$ for an appropriate $n.$ 

As the action is free it follows from Theorem \ref{bredon2} that this sequence does not collapse in the $E_{2}-$term, otherwise there would be infinite nonzero elements in the orbit space cohomology. So, there are $r_{i}\geq 2$, for which some differentials $d_{r_{i}}^{p,q}:E_{r_{i}}^{p,q}\to E_{r_{i}}^{p+r_{i},q+1-r_{i}}$ are non-trivial. If $r\geq 2$ is the smaller integer among the $r_{i}'$s, then $E_{2}=\cdots =E_{r},$ because $d_{s}=0,$ for any $2\leq s\leq r-1.$ Therefore, 
\begin{equation}
E_{s+1}^{p,q}=\ker(d_{s}^{p,q})/\text{im}(d_{s}^{p-s,q+s-1})=E_{s}^{p,q}/\{0\}=E_{s}^{p,q},
\end{equation} 
for all $p,q\geq 0.$ 

In order to analyze the possible cases when $d_{r}$ is nontrivial, as $x,c$ and $d$ are the generators of $H^{\ast}(X),$ it is enough to look at the actions under such elements, that is, we will look at the possible values to $d_{r}^{0,1}(1\otimes x),$ $d_{r}^{0,1}(1\otimes c)$ and $d_{r}^{0,2}(1\otimes d).$  

Knowing that the sequence is a first quadrant spectral sequence, moreover $\text{im}(d_{r}^{0,1})\subseteq E_{r}^{r,2-r}$ and $\text{im}(d_{r}^{0,2})\subseteq E_{r}^{r,3-r},$ then these differentials can be non trivial only when $r=2$ or $r=3.$ However, at $E_{3}-$page it is immediate that $d_{3}^{0,1}(1\otimes x)=d_{3}^{0,1}(1\otimes c)=0,$ 
which implies that only $d_{3}^{0,2}(1\otimes d)$ can be nonzero. So, at first, we should analyze the following cases:

\begin{itemize}
\item[(A)] $d_{3}^{0,1}(1\otimes x)=0,$ $d_{3}^{0,1}(1\otimes c)=0$ and $d_{3}^{0,2}(1\otimes d)\neq 0;$
\item[(B)] $d_{2}^{0,1}(1\otimes x)=0,$ $d_{2}^{0,1}(1\otimes c)= 0$ and $d_{2}^{0,2}(1\otimes d)\neq 0;$
\item[(C)] $d_{2}^{0,1}(1\otimes x)\neq 0,$ $d_{2}^{0,1}(1\otimes c)\neq 0$ and $d_{2}^{0,2}(1\otimes d)= 0;$
\item[(D)] $d_{2}^{0,1}(1\otimes x)\neq 0,$ $d_{2}^{0,1}(1\otimes c)\neq 0$ and $d_{2}^{0,2}(1\otimes d)\neq 0;$
\item[(E)] $d_{2}^{0,1}(1\otimes x)\neq 0,$ $d_{2}^{0,1}(1\otimes c)=0$ and $d_{2}^{0,2}(1\otimes d)= 0;$ 
\item[(F)] $d_{2}^{0,1}(1\otimes x)\neq 0,$ $d_{2}^{0,1}(1\otimes c)=0$ and $d_{2}^{0,2}(1\otimes d)\neq 0;$
\item[(G)] $d_{2}^{0,1}(1\otimes x)= 0,$ $d_{2}^{0,1}(1\otimes c)\neq 0$ and $d_{2}^{0,2}(1\otimes d)\neq 0;$
\item[(H)] $d_{2}^{0,1}(1\otimes x)=0,$ $d_{2}^{0,1}(1\otimes c)\neq 0$ and $d_{2}^{0,2}(1\otimes d)= 0.$
\end{itemize}

We will show that the only possible case is $A$, which means the uniqueness of the cohomology ring. For the cases $C$ to $H$, the obs\-tructions are simple to check and we will do it at Proposition \ref{ppp}, using only the relations on the cohomology ring and the multiplicative proprieties of the differentials. To show that $B$ is not possible, more calculations are needed, which will be done at Proposition \ref{pppp}.

\begin{lemma}\label{lll}
If $k$ is even, then $d_{i}^{l,k}(t^{l}\otimes c^{k})=0,$ $d_{i}^{l,2k}(t^{l}\otimes d^{k})=0$ and $d_{i}^{l,2k+1}(t^{l}\otimes \nu d^{k})=0,$ for $\nu$ equal to $c$ or $x,$ for all $i\geq 2.$ 
\end{lemma}

\begin{proof}
If $k=2j,$ note that $d_{i}^{0,k}(1\otimes c^{k})=2 (1\otimes c^{j})d_{i}^{0,j}(1\otimes c^{j})=0,$ then 
$$d_{i}^{l,k}(t^{l}\otimes c^{k})=d_{i}^{l,k}((t^{l}\otimes 1)(1\otimes c^{k}))=2(t^{l}\otimes 1)(1\otimes c^{j})d_{i}^{0,j}(1\otimes c^{j})=0.$$ 

The other cases are similar.
\end{proof}

\begin{prop}\label{ppp}
The cases $C$ to $H$ cannot occurs.
\end{prop}

\begin{proof}
For an obstruction to the cases $E$ and $F,$ let us observe that it is impossible to occur simultaneously $d_{2}^{0,1}(1\otimes x)\neq 0$ and $d_{2}^{0,1}(1\otimes c)=0,$ since in this case (using Lemma \ref{lll} in its simplest form to $k=2$) we have $t^{2}\otimes c=d_{2}^{0,2}(1\otimes cx)=d_{2}^{0,2}(1\otimes c^{2})=0.$ Similarly, we find obstructions for the cases $C,D,G$ and $H.$ 
\end{proof}

\begin{rem}
For $D$ we still have another obstruction. Suppose that $d_{2}^{0,1}(1\otimes x)=d_{2}^{0,1}(1\otimes c)=t^{2}\otimes 1\neq 0.$ If 
$$d_{2}^{0,2}(1\otimes d)\in E_{2}^{2,1}\cong H^{2}(B_{\Z_{2}})\otimes H^{1}(X)$$ 
is nonzero, then it can only be equal to $t^{2}\otimes(x+c).$ 

In fact, as $d_{2}^{2,1}(t^{2}\otimes x)=d_{2}^{2,1}(t^{2}\otimes c)=t^{4}\otimes 1,$ then this is the only way to 
$$\text{im}(d_{2}^{0,2})\subseteq\ker(d_{2}^{2,1})=\langle t^{2}\otimes (x+c) \rangle.$$ 

Now, Let us observe that $d_{2}^{0,2}(1\otimes cx)=t^{2}\otimes (x+c),$
that is,\linebreak 	$d_{2}^{0,2}(1\otimes d)=d_{2}^{0,2}(1\otimes cx).$ On the other hand, 
$$t^{2}\otimes(x+c)=d_{2}^{0,2}(1\otimes d)=d_{2}^{0,2}(1\otimes cx)=d_{2}^{0,2}(1\otimes c^{2})=0,$$ 
which is a contradiction.
\end{rem}

\begin{prop}\label{pppp}
The case $B$ cannot occurs.
\end{prop}

\begin{proof}
In this case, $d_{2}^{0,2}(1\otimes d)$ is a nonzero element in $E_{2}^{2,1}\cong H^{2}(B_{\Z_{2}})\otimes H^{1}(X),$ so we have the following possibilities:
\begin{itemize}
\item[(B1)] $d_{2}^{0,2}(1\otimes d)=t^{2}\otimes x;$
\item[(B2)] $d_{2}^{0,2}(1\otimes d)=t^{2}\otimes c;$
\item[(B3)] $d_{2}^{0,2}(1\otimes d)=t^{2}\otimes (x+c).$
\end{itemize}

We will show that these cases provide the same convergent limit for the spectral sequence and besides they cannot occur due to Theorem \ref{bredon2}. 

Let $t^{l}$ be the generator of $H^{l}(B_{\Z_{2}}).$ For the $B1$ case, note that\linebreak 
$d_{2}^{0,3}(1\otimes xd)=t^{2}\otimes x^{2}=0,$ $d_{2}^{0,3}(1\otimes cd)=t^{2}\otimes cx,$
and in general way  we have
$$d_{2}^{l,2k}(t^{l}\otimes d^{k})=(t^{l}\otimes 1)d_{2}^{0,2k}(1\otimes d^{k})=\left\{\begin{array}{ccl}
0,                     &\text{ if } & k\equiv 0(\text{mod }2),\\
t^{l+2}\otimes d^{k-1}x, &\text{ if } & k\equiv 1(\text{mod }2),
\end{array}\right.$$
$$d_{2}^{l,2k+1}(t^{l}\otimes cd^{k})=\left\{\begin{array}{ccl}
0,                     &\text{ if } & k\equiv 0(\text{mod }2),\\
t^{l+2}\otimes d^{k-1}cx, &\text{ if } & k\equiv 1(\text{mod }2),
\end{array}\right.$$
 $d_{2}^{l,2k+1}(t^{l}\otimes xd^{k})=0$ and $d_{2}^{l,2k+2}(t^{l}\otimes c^{2}d^{k})=0,$ for all $l\geq 0$ and $k\in\{0,\cdots,n\}.$ With this complete description of the differentials the $E_{3}-$term has the following pattern for $q\leq 2n$:
\begin{equation}\label{e3}
E_{3}^{p,q}=\left\{\begin{array}{rl}
\Z_{2},    &\text{ if } q=0\text{ or }q\equiv 3(\text{mod }4),\\
\Z_{2},    &\text{ if }q\equiv 1(\text{mod }4)\text{ and }p\geq 2,\\
\Z_{2}^{2},&\text{ if } q\neq 0\text{ and }q\equiv 0(\text{mod }4),\\
\Z_{2}^{2},&\text{ if }q\equiv 1(\text{mod }4)\text{ and }p=0,1,\\
\{0\},     &\text{ if }q\equiv 2(\text{mod }4)\text{ and }p\geq 2.
\end{array}\right.
\end{equation}

We affirm that all the differentials $d_{3}^{p,q}$ are null. 

Indeed, if $q\equiv 0(\text{mod }4),$ then we have $q-2\equiv 2(\text{mod }4)$ and\linebreak $\text{im} d_{3}^{p,q}\subseteq E_{3}^{p+3,q-2}=\{0\},$ so $d_{3}^{p,q}$ is trivial.

If $q\equiv 1(\text{mod }4),$ note that 
$$E_{3}^{p,q}=\langle t^{p}\otimes cd^{(q-1)/2} \rangle \oplus \langle t^{p}\otimes xd^{(q-1)/2} \rangle,$$ 
for $p=0,1$ and $E_{3}^{p,q}=\langle t^{p}\otimes cd^{(q-1)/2} \rangle,$ for $p\geq 2.$ In this case $(q-1)/2$ is even, therefore, by Lemma \ref{lll}, in both cases for $\nu$ equal to $x$ or $\nu$ equal to $c,$ we have $d_{3}^{p,q}(t^{p}\otimes \nu d^{(q-1)/2})=0,$ so $d_{3}^{p,q}$ is trivial.

If $q\equiv 2(\text{mod }4),$ then $E_{3}^{p,q}=\langle t^{p}\otimes c^{2}d^{(q-2)/2}\rangle,$ for $p=0,1,$ and it is trivial for $p\geq 2.$ Again, by the parity of $(q-2)/2,$ we have $d_{3}^{p,q}=0.$

If $q\equiv 3(\text{mod }4),$ then $E_{3}^{p,q}=\langle t^{p}\otimes xd^{(q-1)/2}\rangle$ and the analysis of $d_{3}^{p,q}$ depends only on of $d_{3}^{0,q-1},$ since
$$d_{3}^{p,q}\left(t^{p}\otimes xd^{(q-1)/2}\right)=\left(t^{p}\otimes x\right)\cdot d_{3}^{0,q-1}\left(1\otimes d^{(q-1)/2}\right).$$

We affirm that $d_{3}^{0,q-1}(1\otimes d^{(q-1)/2})=0.$ In fact, as $1\otimes d$ represents the zero element in 
$$E_{3}^{0,2}=(\ker d_{2}^{0,2}/\text{im}d_{2}^{-2,2})=\langle 1\otimes c^{2}\rangle \cong (\langle 1\otimes c^{2}\rangle \oplus \langle 1\otimes d \rangle)/\langle 1\otimes d \rangle,$$ 
we have $d_{3}^{0,2}(1\otimes d)=0,$ and since $q-3\equiv 0(\text{mod }4),$ we concluded that $d_{3}^{0,q-3}$ is trivial. These facts together ensure that
$$
\begin{array}{rcl}
d_{3}^{0,q-1}(1\otimes d^{(q-1)/2})&=& d_{3}^{0,q-1}((1\otimes d)\cdot(1\otimes d^{(q-3)/2}))\\
                                   &=& \underbrace{d_{3}^{0,2}(1\otimes d)}_{0}(1\otimes d^{(q-3)/2})+(1\otimes d)\underbrace{d_{3}^{0,q-3}(1\otimes d^{(q-3)/2})}_{0}\\
																	 &=& 0.

\end{array}
$$

 For $B2,$ let us note that the differentials are given by the rules 
$$d_{2}^{l,2k}(t^{l}\otimes d^{k})=\left\{\begin{array}{ccl}
0,                     &\text{ if } & k\equiv 0(\text{mod }2),\\
t^{l+2}\otimes d^{k-1}c, &\text{ if } & k\equiv 1(\text{mod }2),
\end{array}\right.
$$
$$d_{2}^{l,2k+1}(t^{l}\otimes cd^{k})=d_{2}^{l,2k+1}(t^{l}\otimes xd^{k})=\left\{\begin{array}{ccl}
0,                     &\text{ if } & k\equiv 0(\text{mod }2),\\
t^{l+2}\otimes d^{k-1}c^{2}, &\text{ if } & k\equiv 1(\text{mod }2),
\end{array}\right.
$$
and $d_{2}^{l,2k+2}(t^{l}\otimes c^{2}d^{k})=0,$ for all $k\in\{0,\cdots,n\}.$ Therefore, the $E_{3}-$term of the spectral sequence has the same pattern of the case $B1,$ given in (\ref{e3}), and similarly, all the differentials $d_{3}^{p,q}$ are null.

For the case $B3,$ the differentials are given by 
$$d_{2}^{l,2k}(t^{l}\otimes d^{k})=\left\{\begin{array}{ccl}
0,                          &\text{ if }& k\equiv 0(\text{mod }2),\\
t^{l+2}\otimes (x+c)d^{k-1},&\text{ if }& k\equiv 1(\text{mod }2),\\
\end{array}\right.$$
$$d_{2}^{l,2k+1}(t^{l}\otimes xd^{k})=d_{2}^{l,2k+1}(t^{l}\otimes cd^{k})=\left\{\begin{array}{ccl}
0,                          &\text{ if }& k\equiv 0(\text{mod }2),\\
t^{l+2}\otimes c^{2}d^{k-1},&\text{ if }& k\equiv 1(\text{mod }2),\\
\end{array}\right.$$ 
and $d_{2}^{l,2k+2}(t^{l}\otimes c^{2}d^{k})=0$ for all $k\in\{0,\cdots,n\}.$ Then, the $E_{3}-$term also has the same pattern as in (\ref{e3}), and similarly we show that the all differentials $d_{3}^{p,q}$ are null.

Since all the differentials in these cases are trivial, this means that the sequence collapses at $E_{3}-$term, so $H^{j}(X/G)\cong \text{Tot}(E_{\infty})^{j}=\text{Tot}(E_{3})^{j},$ for all $j\geq 0,$ which implies that the cohomology of the orbit space have nonzero elements in all dimensions, which is a contradiction with Theorem \ref{bredon2}.
\end{proof}

From the previous results, we have only one possible cohomology ring structure for the orbit space $X/G,$ which is provided by the rule of the differentials in the case $A$ of the spectral sequence. 

Note that, under the hypothesis that $\pi_{1}(B_{G})$ acts trivially on\linebreak $H^{\ast}(Q(m,n)),$ then the same arguments of Proposition \ref{ppp} can be used to show that the cases $C$ to $F$ cannot occurs in this more general situation, due Lemma \ref{lll}, Moreover the tree possibles cases $B1,$ $B2$ and $B3$ also produce the same limit page in the spectral sequence. This means that in those cases we will have at most two possibles cohomology ring structures for the orbit space $Q(m,n)/G,$ given by similar cases to $A$ and $B.$ 


\section{Proof of theorem 1.1}\label{Section5}


In this section, we calculate the cohomology of the total space $X_{G},$ where $G=\Z_{2}$ and $X=Q(1,n),$ for $n\geq 3$ odd, which give us the cohomology algebra of the orbit space $X/G$ by the isomorphism (\ref{is}). By Propositions \ref{ppp} and \ref{pppp} we only need to assume the case $A,$ where $d_{3}^{0,1}(1\otimes x)=d_{3}^{0,1}(1\otimes c)=0$ and $d_{3}^{0,2}(1\otimes d)=t^{3}\otimes 1.$ In this case, for $t^{l}$ the generator of $H^{l}(B_{G})$ and $k\in\{0,\cdots,n\},$ we have that 
$$d_{3}^{l,2k}(t^{l}\otimes d^{k})=\left\{
\begin{array}{ccc}
t^{l+3}\otimes d^{k-1}, &\text{ if }& k\equiv 1(\text{mod }2),\\
0,                      &\text{ if }& k\equiv 0(\text{mod }2),
\end{array}\right. ,$$ 
$$
d_{3}^{l,2k+1}(t^{l}\otimes xd^{k})=\left\{\begin{array}{ccc}
t^{l+3}\otimes xd^{k-1}, &\text{ if }& k\equiv 1(\text{mod }2),\\
0,                      &\text{ if }& k\equiv 0(\text{mod }2),
\end{array}\right. $$ 
$$
d_{3}^{l,2k+i}(t^{l}\otimes c^{i}d^{k}) =\left\{\begin{array}{ccc}
t^{l+3}\otimes c^{i}d^{k-1}, &\text{ if }& k\equiv 1(\text{mod }2),\\
0,                      &\text{ if }& k\equiv 0(\text{mod }2),
\end{array}\right. \text{ for }i=1,2.$$

Thus, we get an explicit expression for the differentials $d_{3}^{p,q},$ which enables us to obtain the $E_{4}-$term of the spectral sequence, which have the following pattern, for $q\leq 2n $:
$$ E_{4}^{p,q}=\left\{\begin{array}{rl}
\Z_{2}, & \text{ if } q\equiv 0(\text{mod }4) \text{ or }q\equiv 2(\text{mod }4), \\
\Z_{2}^{2}, & \text{ if }q\equiv 1(\text{mod }4),\\
\{0\},  & \text{ if }q\equiv 3(\text{mod }4) ,
\end{array}\right.\text{ for }p=0,1,2,$$
and $E_{4}^{p,q}=\{0\},$ for all $p\geq 3.$

Therefore, $\text{im}d_{4}^{p,q}\subseteq E_{4}^{p+4,q-3}=\{0\}$ and the sequence collapses at $E_{4}-$term, that is, $E_{\infty}\cong E_{4},$ and 
$$H^{j}(X/G)\cong H^{j}(X_{G})\cong \text{Tot}(E_{\infty})^{j}=\bigoplus_{p+q=j}E_{\infty}^{p,q}= E_{\infty}^{0,j}\oplus E_{\infty}^{1,j-1}\oplus E_{\infty}^{2,j-2},
$$	
or, more specifically, we have the following additive structure of $H^{\ast}(X/G),$
$$H^{j}(X/G)\cong \left\{\begin{array}{ccl}
\Z_{2},&\text{ if }& j=0\text{ or }j=2n+2,\\
\Z_{2}^{2},&\text{ if }&j\equiv 0(\text{mod }4),\text{ for } 0<j<2n+2,\\
\Z_{2}^{3},&\text{ if }&j\equiv 1(\text{mod }4),\text{ for } 0<j<2n+2,\\
\Z_{2}^{4},&\text{ if }&j\equiv 2(\text{mod }4),\text{ for } 0<j<2n+2,\\
\Z_{2}^{3},&\text{ if }&j\equiv 3(\text{mod }4),\text{ for } 0<j<2n+2,\\
0,         &\text{ if }&j>2n+2.
\end{array}\right.$$	

To obtain the ring structure of the orbit space, we will use the homomorphisms $\pi^{\ast}$ and $i^{\ast},$ as given by (\ref{hom1}) and (\ref{hom2})).

Let us consider $\alpha=t\otimes 1\in E^{1,0}_{\infty}.$ By (\ref{hom1}), we have $\pi^{\ast}(t)=\alpha\in H^{1}(X_{G})$ and $\alpha^{l}=\pi^{\ast}(t^{l})\in E^{l,0}_{\infty},$  that is, $\alpha^{l}=0,$ for $l\geq 3$ and $\alpha^{l}\neq 0,$ for $l=1,2.$

The elements $1\otimes c,\, 1\otimes x \in E_{2}^{0,1}$ and $1\otimes d^{2}\in E_{2}^{0,4}$ are permanent co-cycles and they determine nonzero elements $y,z\in E_{\infty}^{0,1}$ and $w\in E_{\infty}^{0,4},$ respectively, with $y^{l}=z^{l}=0,$ for $l\geq 3,$ and $w^{k}=0,$ for $k\geq (n+1)/2.$

By (\ref{hom2}), there are elements $\beta,\gamma\in H^{1}(X_{G})$ that represent $y$ and $z,$ respectively, such that $i^{\ast}(\beta)=c$ and $i^{\ast}(\gamma)=x,$ where $i^{\ast}$ is given by

\begin{equation}
H^{1}(X_{G})\twoheadrightarrow E_{\infty}^{0,1}\cong E_{4}^{0,1}\cong E_{3}^{0,1}\cong E_{2}^{0,1}\cong H^{1}(X).
\end{equation}

Since $H^{1}(X_{G})\cong \Z_{2}^{3}$ and $i^{\ast}(\alpha)=i^{\ast}	\circ \pi^{\ast}(t)=0,$ then $\alpha\neq \beta\neq \gamma$ and 
\begin{equation}
H^{1}(X_{G})\cong \langle\alpha \rangle\oplus \langle \beta\rangle\oplus  \langle\gamma \rangle.
\end{equation} 

In the $E_{\infty}-$term we have $z^{2}=0,$ $y^{2}=zy$ and $y^{2}z\in E_{\infty}^{0,3}=\{0\},$ then $\gamma^{2}=0$ and $\beta^{2}=\beta\gamma.$ Therefore, 
\begin{equation}
H^{2}(X_{G})\cong \langle \alpha^{2}\rangle\oplus \langle \alpha\beta\rangle \oplus\langle \alpha\gamma\rangle \oplus \langle \beta\gamma=\beta^{2}\rangle.
\end{equation} 

Similarly, we have 
\begin{equation}
H^{3}(X_{G})\cong \langle \alpha\beta^{2}=\alpha\beta\gamma \rangle\oplus\langle \alpha^{2}\beta\rangle \oplus \langle \alpha^{2}\gamma\rangle,
\end{equation} 
where $\beta^{2}\gamma=0.$ 

Again by (\ref{hom2}), there is an element $\delta\in H^{4}(X_{G})\cong \Z_{2}^{2}$ which re\-presents $w$ and such that $i^{\ast}(\delta)=d^{2},$ since $1\otimes dc^{2}\in E_{3}^{0,4}$ it is not a permanent co-cycle and $i^{\ast}$ is factored through the inclusions as follows, 
\begin{equation}\label{inc1}
H^{4}(X_{G})\twoheadrightarrow E_{\infty}^{0,4}\cong E_{4}^{0,4}\subseteq E_{3}^{0,4}\cong E_{2}^{0,4}\cong H^{4}(X),
\end{equation}
therefore, 
\begin{equation}
H^{4}(X_{G})\cong \langle \alpha^{2}\beta^{2}=\alpha^{2}\beta\gamma\rangle \oplus \langle \delta\rangle, 
\end{equation}
and $\delta^{\frac{n+1}{2}}=0.$

Let us also observe that $w^{\frac{n-1}{2}} x^{2}y^{2}=w^{\frac{n-1}{2}} x^{2}yz\in E_{\infty}^{2,2n}\neq \{0\}$ 
and $E_{\infty}^{l,m}=\{0\},$ for all $l\geq 0$ and $m\geq 2n+1,$ that is, in general $\delta^{i}\alpha^{j}\beta\gamma\neq 0,$ for all $0\leq i\leq (n-1)/2$ and $j=1,2,$ with the generator $\delta^{\frac{n-1}{2}}\alpha^{2}\beta^{2}$ in the top dimension. Thus, $H^{\ast}(X/G)$ is isomorphic to the following polynomial graded algebra
$$\frac{\Z_{2}[\alpha,\beta,\gamma,\delta]}{\langle \alpha^{3}, \beta^{3}, \gamma^{2},\beta^{2}+\beta\gamma, \delta^{\frac{n+1}{2}}\rangle},$$ 
which completes the proof of Theorem \ref{main}.



\end{document}